\documentclass[reqno, letterpaper, oneside]{article}
\usepackage[utf8]{inputenc}
\usepackage{amsmath, amsthm, amsfonts,amssymb}
\usepackage{geometry}
\usepackage{bbm}

\newcommand\E{\operatorname{\mathbb E{}}}
\renewcommand\Pr{\operatorname{\mathbb P{}}}

\newcommand\Bin{\operatorname{Bin}}

\newcommand{\eps}{\epsilon}
\newcommand{\cD}{\mathcal{D}}
\newcommand{\cE}{\mathcal{E}}
\newcommand{\cF}{\mathcal{F}}
\newcommand{\cG}{\mathcal{G}}
\newcommand{\cH}{\mathcal{H}}

\newcommand{\cJ}{\mathcal{J}}

\newtheorem{theorem}{Theorem}
\newtheorem{lemma}[theorem]{Lemma}

\newtheorem{remark}{Remark}

\newcommand\bigpar[1]{\bigl(#1\bigr)}
\newcommand\Bigpar[1]{\Bigl(#1\Bigr)}
\newcommand\biggpar[1]{\biggl(#1\biggr)}

\newcommand\Bigsqpar[1]{\Bigl[#1\Bigr]}

\newcommand\ceil[1]{\lceil#1\rceil}

\newcommand\floor[1]{\lfloor#1\rfloor}

\newcommand{\indic}[1]{\mathbbm{1}_{\{{#1}\}}}

\newcommand{\pc}{p_{\mathrm c}} 
\newcommand{\qf}{q_{\mathrm f}} 

\let\OLDthebibliography\thebibliography
\renewcommand\thebibliography[1]{
	\OLDthebibliography{#1}
	\setlength{\parskip}{0pt}
	\setlength{\itemsep}{0pt plus 0.3ex}
}

\title{Note on down-set thresholds}
\author{Lutz Warnke%
\thanks{Department of Mathematics, University of California, San Diego, La Jolla CA~92093, USA. 
E-mail: {\tt lwarnke@ucsd.edu}. 
Supported by NSF~CAREER grant~DMS-2225631, and a Sloan Research Fellowship.}}
\date{April~26, 2023; revised October~24, 2023}
\begin{document}

\maketitle

\begin{abstract}
Gunby--He--Narayanan showed that the logarithmic gap predictions of Kahn--Kalai and Talagrand 
(proved by Park--Pham and Frankston--Kahn--Narayanan--Park) about thresholds of up-sets do not apply to down-sets.  
In particular, for the down-set of triangle-free graphs, they showed  that there is a polynomial gap between the threshold and the factional expectation threshold. 
In this short note we give a simpler proof of this result, 
and extend the polynomial threshold gap to down-sets of~$F$-free~graphs.
\end{abstract}

\section{Introduction}
Thresholds are a central theme in the theory of random discrete structures, ever since the seminal work of Erd\H{o}s and Renyi~\cite{ER1960}. 
The basic idea is that a threshold corresponds to the density at which a random set or random graph changes from unlikely satisfying to likely satisfying some property of interest.
The general thresholds predictions of Kahn--Kalai~\cite{KK} and Talagrand~\cite{TL}, 
which were proved in recent breakthroughs by Park and Pham~\cite{PP} and Frankston, Kahn, Narayanan and Park~\cite{FKNP}, 
show that the (fractional) expectation threshold differs by at most a logarithmic factor from the threshold for up-sets, i.e., monotone increasing properties. 
The conceptual point is that these logarithmic threshold gaps are useful in both theory and applications, since the expectation thresholds are usually much easier to locate than the threshold~itself. 

Recently, Gunby, He and Narayanan~\cite{GHN2021} demonstrated that these logarithmic gaps do not carry over to the thresholds of down-sets, i.e., monotone decreasing properties. 
Indeed, the main result of~\cite{GHN2021} is the special case of Theorem~\ref{thm:main} below where~$F=\Delta$ is a triangle: 
it states that, for the down-set of triangle-free graphs, there is a polynomial gap between the threshold and the factional expectation~threshold. 
The proof in~\cite{GHN2021} 
is based on a non-trivial large-deviation machinery that is specific to triangles. 
Given the importance of thresholds in the theory of random discrete structures,
it is desirable to have simple proofs for such fundamental~results. 

The aim of this short note is to record 
a much simpler proof of a more general polynomial gap result for the thresholds of down-sets: 
Theorem~\ref{thm:main} shows that, for the down-set of~\mbox{$F$-free} graphs, there is a polynomial gap between the threshold and the factional expectation threshold 
for essentially all non-empty graphs~$F$ of interest (the only excluded graphs~$F$ consist of a matching plus potentially some isolated~vertices). 
\begin{theorem}[Main result]\label{thm:main}
Given a graph~$F$ with maximum degree at least two, let~$\cF_n$ be the down-set of~$F$-free graphs on vertex-set~$[n]$.
Then the expectation threshold~$q(\cF_n)$ and the fractional expectation-threshold~$\qf(\cF_n)$ of~$\cF_n$ 
are asymptotically both a polynomial factor larger than the threshold~$\pc(\cF_n)$ of~$\cF_n$: 
there are constants~${\gamma=\gamma(F)>0}$ and~${n_0=n_0(F) \ge 1}$ such that~$q(\cF_n),\qf(\cF_n) \ge n^{\gamma} \pc(\cF_n)$ for all~$n \ge n_0$. 
\end{theorem}
\begin{remark}\label{rem:main}
The proof shows that~$q(\cF_n),\qf(\cF_n) = \Omega(n^{-1/m_2(F)})$ and~$\pc(\cF_n) = \Theta(n^{-1/m(F)})$,  
with the graph-densities~$m_2(F):=\max_{J\subseteq F: v_J \ge 3}\tfrac{e_J-1}{v_J-2}$ and~$m(F):=\max_{J\subseteq F: v_J \ge 1}\tfrac{e_J}{v_J}$. 
\end{remark}

Gearing up towards proofs, we now introduce some standard threshold definitions and terminology from~\cite{KK,TL,DK,FKNP,FKP,PP,GHN2021}. 
Given a finite set~$X$ and an inclusion probability~$p \in [0,1]$, 
we write~$\mu_p$ for the product measure on the power set~$2^X$ of~$X$ given by~${\mu_p(S) = p^{|S|}(1-p)^{|X\setminus S|}}$ for all~$S \subseteq X$. 
Given a family~$\cF \subseteq X$, we call~$\cF$ a down-set if it is closed under taking subsets (i.e., if~$A \in \cF$ and~$B \subseteq A$ then~$B \in \cF$), 
and~$\cF$ an up-set if it is closed under taking supersets (i.e., if~$A \in \cF$ and~$B \supseteq A$ then~$B \in \cF$). 
Excluding the uninteresting cases~$\cF=\emptyset$ and~$\cF=2^{\cF}$, for any down-set or up-set~$\cF$ the \emph{threshold}~$\pc(\cF)$ is defined as the unique~$p \in [0,1]$ for which~$\mu_p(\cF)=1/2$ 
(this is well-defined, since~$\mu_p(\cF)=\sum_{S \in \cF}\mu_p(S)$ is 
 strictly monotone in~$p$). 
For~$X=\binom{[n]}{2}$ and the down-set~$\cF_n$ of~$F$-free graphs on vertex-set~$[n]$ considered in Theorem~\ref{thm:main},
the `small subgraphs' random graphs result of Bollob\'{a}s~\cite{SGTR} from~1984 
shows that the threshold of~$\cF_n$~satisfies
\begin{equation}\label{eq:pc}
\pc(\cF_n) = \Theta(n^{-1/m(F)}).
\end{equation} 

Hence the real content of Theorem~\ref{thm:main} and Remark~\ref{rem:main} are the bounds~$q(\cF_n),\qf(\cF_n) = \Omega(n^{-1/m_2(F)})$ on the expectation threshold and fractional expectation threshold. 
We shall establish these in the following two~subsections, 
where for clarity of exposition we will discuss both thresholds separately, in increasing level of~generality
(since~$q(\cF_n) \ge \qf(\cF_n)$ holds, as we shall see in Section~\ref{sec:fractional}).

\subsection{Expectation Threshold: Polynomial Gap}\label{sec:expectation}
The \emph{expectation threshold}~$q(\cF)$ of a down-set~$\cF$ is defined as the smallest value of~$p \in [0,1]$ for which there exists a `certificate'~$\cG \subseteq 2^X$ 
with~$\sum_{S \in \cG}(1-p)^{|X \setminus S|} \le 1/2$ and~$\cF \subseteq \cG^{^{\downarrow}}$, 
where~$\cG^{\downarrow}:=\bigcup_{S \in \cG}\{T : T \subseteq S\}$ is the decreasing family generated by~$\cG$; see~\cite{GHN2021}. 
Although not important for this note, the crux of this definition is that~$\pc(\cF) \le q(\cF)$ follows immediately,  
since for~$p=q(\cF)$ we have the simple `first moment' upper~bound 
\[
\mu_p(\cF) \le \mu_p(\cG^{^{\downarrow}}) \le \sum_{S \in \cG} \mu_p(\{T: T \subseteq S\}) = \sum_{S \in \cG}(1-p)^{|X \setminus S|} \le 1/2 .
\]

For~$X=\binom{[n]}{2}$ and the down-set~$\cF_n$ of~$F$-free graphs on vertex-set~$[n]$ considered in Theorem~\ref{thm:main},
we now claim 
that our main technical result (Lemma~\ref{lem:main:technical} below) implies the expectation threshold lower bound
\begin{equation}\label{eq:qc}
q(\cF_n) = \Omega(n^{-1/m_2(F)}) .
\end{equation} 
Indeed, writing~$q=q(\cF_n$), note that by definition there exists a certificate~$\cG$ consisting of a collection of graphs on vertex-set~$[n]$ with~$\cF_n \subseteq \cG^{^{\downarrow}}$ and~$\sum_{J \in \cG}(1-q)^{\binom{n}{2}-e(J)} \le 1/2$. 
Using the constants~$\eps,\delta \in (0,1/2]$ from Lemma~\ref{lem:main:technical}, 
suppose for the sake of contradiction that~$q \le 1-e^{-\delta p}$ for~$p:=\eps n^{-1/m_2(H)}$. 
Defining~$\cH$ as the collection of complements of the graphs in~$\cG$ (where we take complements with respect to the edge-set), 
we then have
\[
\sum_{H \in \cH}\exp(-\delta e(H)p) \le \sum_{J \in \cG}(1-q)^{\binom{n}{2}-e(J)} \le 1/2.
\] 
By Lemma~\ref{lem:main:technical} there thus is an $F$-free graph~$G \in \cF_n$ that shares at least one edge with every~$H\in\cH$, 
which by construction means that~$G \in \cF_n$ shares at least one non-edge with every graph in~$\cG$.
Recalling that~$\cG^{\downarrow}=\bigcup_{J \in \cG}\{F: F \subseteq J\}$ contains only subgraphs of graphs in~$\cG$, 
we infer that~$G \notin \cG^{\downarrow}$.
Since this contradicts~$\cF_n \subseteq \cG^{^{\downarrow}}$, 
using~$\delta p \le \delta \eps \le 1/4$ we readily infer that
\[
q(\cF_n) =q \ge 1-e^{-\delta p} \ge \delta p/2=\delta \eps/2 \cdot n^{-1/m_2(H)},
\] 
establishing the claimed lower bound~\eqref{eq:qc}.
\begin{lemma}[Main technical result]\label{lem:main:technical}%
For any graph~$F$ with maximum degree at least two, 
there are constants ${\eps \in (0,1/2]}$ and $\delta=\delta(F) \in (0,1/2]$ such that the following holds for all~$n \ge 1$ and $0 < p \le \eps n^{-1/m_2(F)}$. 
If~$\cH$ is a collection of graphs on vertex-set~$[n]$~with
\begin{equation}\label{equation:prob-assumption}
\sum_{H\in\cH}\exp\Bigpar{-\delta e(H)p} \le 1/2, 
\end{equation}
then there is an $F$-free graph~$G$ on vertex-set~$[n]$ that shares at least one edge with every~$H\in\cH$.
\end{lemma}
In~\cite{GHN2021} a proof of Lemma~\ref{lem:main:technical} for triangles~$F=\Delta$ is given:
it is based on a refinement of a clever random graph alteration argument of Erd{\H{o}}s~\cite{erdos1961graph} from~1961, 
and requires the development of some non-trivial triangle-specific large-deviation machinery  
(see the proof of Theorem~1.4 in~\cite{GHN2021}, which is spread across Sections~2--3 therein). 
Our simpler and shorter proof for arbitrary~$F$ is based on a refinement of an elegant random graph alteration argument of Krivelevich~\cite{krivelevich1995bounding} from~1994 
(which was recently further developed in~\cite{GW2019}). 
\begin{proof}[Proof of Lemma~\ref{lem:main:technical}]
With foresight, set~$\eps:= \min_{x \ge 2}[1/(3xe^2)]^{1/(x-1)}=1/(6e^2)$ and~$\delta := 1/\max\{9,4e_F\}$, 
so that~$n \ge 1$ implies~${p \le \eps < 1}$. 
Furthermore, we pick a minimal subgraph~$J \subseteq F$ with~$m_2(J)=m_2(F)$. 
Note that~$e_J \ge 2$, since~$F$ contains a vertex of degree at least two.
Let~$\cJ_{n,p}$ be a size-maximal collection of edge-disjoint $J$-copies in~$G_{n,p}$. 
We construct~${G_{n,p}(J) \subseteq G_{n,p}}$ by removing all edges from~$G_{n,p}$ that are in some~$J$-copy from~$\cJ_{n,p}$. 
We will show that~${G_{n,p}(J)}$ satisfies the desired properties with positive probability. 
%
By construction~$G_{n,p}(J)$ is $J$-free (since any~$J$-copy in~$G_{n,p}(J)$ could be added to~$\cJ_{n,p}$, contradicting size-maximality of~$\cJ_{n,p}$), and therefore $F$-free.
Thus, it will suffice to show that with positive probability, $G_{n,p}(J)$ intersects every~$H \in \cH$. 
Intuitively, we will show this by proving that for any given~$H \in \cH$, the random graph~$G_{n,p}$ typically contains many edges in $H$ and~$\cJ_{n,p}$ contains few.

Turning to the details, fix a graph~$H \in \cH$. 
Note that~\eqref{equation:prob-assumption} implies that~$H$ has at least~${e(H) \ge 1}$ edges 
(since~$e(H) =0$ implies that the left-hand side of~\eqref{equation:prob-assumption} is at least one, contradicting our assumption). 
Let~$\cE_H$ denote the event that~$G_{n,p}$ contains at least~$e(H)p/2$ edges from~$H$.
Since the number~$|E(G_{n,p}) \cap E(H)|$ of common edges of~$G_{n,p}$ and~$H$ 
is a binomial random variable with distribution~$\Bin(e(H),p)$, 
using standard Chernoff bounds (such as~\cite[Theorem~2.1]{JLR}) it follows~that 
\begin{equation}
\Pr(\neg\cE_H) 
\le 
\Pr\bigpar{\Bin(e(H),p) \le e(H)p/2}
\le \exp\Bigpar{-e(H)p/8} . 
\end{equation}
Let~$\cD_H$ denote the event that~$\cJ_{n,p}$ contains at most $m:= e(H)p/(3e_J)$ many $J$-copies that share an edge with~$H$. 
Minimality of~$J \subseteq F$ implies~$m_2(F)=m_2(J)=(e_J-1)/(v_J-2)$, so that
\[
{n^{v_J-2}p^{e_J-1} \le \eps^{e_J-1}} \le 1/(3e_Je^2)
\]
 by definition of~$\eps$. 
Note that if~$\cD_H$ fails, then there is a subcollection~$\cJ \subseteq \cJ_{n,p}$ of exactly~${|\cJ|=\ceil{m}}$ many $J$-copies that are all edge-disjoint.
Since there are at most~$e(H)n^{v_J-2}$ possible $J$-copies that can share an edge with~$H$, 
using a standard union bound argument and~$\binom{n}{x} \le (ne/x)^x$ it follows~that 
\begin{equation}
\Pr(\neg\cD_H) 
\le \binom{e(H)n^{v_J-2}}{\ceil{m}}p^{e_J \ceil{m}} 
\le \biggpar{\frac{e(H)n^{v_J-2}p^{e_J}e}{\ceil{m}}}^{\ceil{m}}
\le \biggpar{\frac{e(H)p}{3e_Je\ceil{m}}}^{\ceil{m}} 
\le \exp\Bigpar{-e(H)p/(3e_J)}.
\end{equation}
Note that if~$\cE_H \cap \cD_H$ both hold, then~$G_{n,p}(J)$ contains at least
\[
\ceil{e(H)p/2}-e_J \cdot \floor{m} \ge e(H)p \cdot (1/2-1/3) > 0
\] 
many edges from~$H$ (as~$e(H) \ge 1$ and~$p>0$).
Using a standard union bound argument, it thus follows that the probability that~$G_{n,p}(J)$ does not share an edge with some graph~$H\in\cH$ is at most
\begin{equation}\label{eq:main:upper}
\sum_{H\in\cH} \Bigsqpar{\Pr(\neg\cE_H) + \Pr(\neg\cD_H) } < 2\sum_{H\in\cH} \exp\Bigpar{-\delta e(H)p} \le 1,
\end{equation}
which by the probabilistic method establishes existence of the desired $F$-free graph~$G$. 
\end{proof}

To complete the proof of Theorem~\ref{thm:main} for the expectation threshold~$q(\cF_n)$,
i.e., that we have the polynomial gap~$q(\cF_n) \ge n^{\gamma} \pc(\cF_n)$ for all~$n \ge n_0$, 
in view of estimates~\eqref{eq:pc} and~\eqref{eq:qc} it remains to show that 
\begin{equation}\label{eq:mm2}
m_2(F) > m(F)
\end{equation}  
for any graph~$F$ with maximum degree at least two.
To establish~\eqref{eq:mm2}, we pick a minimal subgraph~$J \subseteq F$ with~$m(J)=m(F)$. 
Since~$F$ contains a vertex of degree at least two, it follows that~$m(J)=\tfrac{e_J}{v_J} \ge 2/3 > 1/2$ and~$e_J \ge 2$ (and thus~$v_J \ge 3$), 
which in turn implies~$m(F) = m(J) = \tfrac{e_J}{v_J} < \tfrac{e_J-1}{v_J-2} \le m_2(F)$, as desired.

\subsection{Fractional Expectation Threshold: Polynomial Gap}\label{sec:fractional}
Viewing the certificate~$\cG$ in the definition of the expectation threshold as an integral map from~$2^X$ to~$\{0,1\}$, 
one naturally arrives at the following fractional relaxation (see also~\cite{TL1995,TL}). 
The \emph{fractional expectation threshold}~$\qf(\cF)$ of a down-set~$\cF$ is defined as the smallest value of~$p \in [0,1]$ for which there exists a `fractional certificate' function~$\lambda:2^X \to [0,\infty)$ 
with~$\sum_{S \subseteq X}\lambda(S)(1-p)^{|X \setminus S|} \le 1/2$ and~$\sum_{S \supseteq F}\lambda(S) \ge 1$ for all~$F \in \cF$; see~\cite{GHN2021}. 
Although not important for this note, it is instructive to note that~$\pc(\cF) \le \qf(\cF) \le q(\cF)$ follows easily: 
for~$\qf(\cF) \le q(\cF)$ this is immediate, 
and for~$\pc(\cF) \le \qf(\cF)$ the crux is that for~$p=\qf(\cF)$ we have 
\[
\mu_p(\cF) 
\le \sum_{F \in \cF}\mu_p(F)\sum_{S \supseteq F}\lambda(S) 
\le \sum_{S \subseteq X}\lambda(S) \mu_p(\{F: F \subseteq S\}) 
= \sum_{S \subseteq X}\lambda(S)(1-p)^{|X \setminus S|} 
\le 1/2. 
\]

Similarly as Lemma~\ref{lem:main:technical} implies the expectation threshold lower bound~\eqref{eq:qc}, 
here Lemma~\ref{lem:main:technical:frac} below (which is a minor modification of Lemma~\ref{lem:main:technical}) implies the fractional expectation threshold lower bound
\begin{equation}\label{eq:qfc}
\qf(\cF_n) = \Omega(n^{-1/m_2(F)}) ,
\end{equation} 
which together with estimates~\eqref{eq:pc} and~\eqref{eq:mm2} then completes the proof of Theorem~\ref{thm:main}. 
\begin{lemma}\label{lem:main:technical:frac}
For any graph~$F$ with maximum degree at least two, 
there are constants ${\eps \in (0,1/2]}$ and $\delta=\delta(F)\in (0,1/2]$ such that the following holds for all~$n \ge 1$ and $0 < p \le \eps n^{-1/m_2(F)}$. 
If~$\cH$ is a collection of graphs on vertex-set~$[n]$ and~$\lambda: \cH \to [0,\infty)$ is a function satisfying
\begin{equation}\label{equation:prob-assumption:frac}
\sum_{H\in\cH}\lambda(H)\exp\Bigpar{-\delta e(H)p} \le 1/2, 
\end{equation}
then there is an $F$-free graph~$G$ on vertex-set~$[n]$ satisfying 
\begin{equation}\label{equation:conclusion:frac}
\sum_{H\in\cH(G)}\lambda(H) < 1, 
\end{equation}
where~$\cH(G)$ is the collection of all graphs~$H \in \cH$ that share no edges with~$G$. 
\end{lemma}
\begin{proof}
This proof is a minor variant of the proof of Lemma~\ref{lem:main:technical}, essentially replacing the union bound by linearity of expectation.
In particular, using the same $F$-free subgraph~${G_{n,p}(J) \subseteq G_{n,p}}$ as before, set
\begin{equation}
X := \sum_{H\in\cH(G_{n,p}(J))}\lambda(H) = \sum_{H \in \cH}\lambda(H) \indic{E(G_{n,p}(J) \cap E(H) = \emptyset},
\end{equation}
where~$\indic{E(G_{n,p}(J) \cap E(H) = \emptyset}$ is the indicator variable for the event 
that~$G_{n,p}(J)$ shares no edges with~$H$. 
The reasoning leading to~\eqref{eq:main:upper} shows that, for any graph~$H \in \cH$ with at least~$e(H) \ge 1$ edges, we~have
\begin{equation}\label{eq:frac:upper:pr}
\Pr\bigpar{E(G_{n,p}(J)) \cap E(H) = \emptyset} \le \Pr(\neg\cE_H) + \Pr(\neg\cD_H) < 2\exp\Bigpar{-\delta e(H)p}.
\end{equation}
In the remaining exceptional case~$e(H)=0$ the right-hand side of~\eqref{eq:frac:upper:pr} equals two, 
which means that this probability upper bound trivially remains valid. 
Using linearity of expectation it thus follows~that 
\begin{equation}
\E X = \sum_{H \in \cH}\lambda(H) \Pr\bigpar{E(G_{n,p}(J)) \cap E(H) = \emptyset} < 2 \sum_{H\in\cH}\lambda(H)\exp\Bigpar{-\delta e(H)p} \le 1,
\end{equation}
which by the probabilistic method establishes existence of the desired $F$-free graph~$G$ satisfying~\eqref{equation:conclusion:frac}. 
\end{proof}

\subsection{Concluding Remarks}\label{sec:conclusion}
On first sight it might be surprising that the (fractional) expectation thresholds~$q(\cF_n),\qf(\cF_n)$ seem more difficult to calculate than the threshold~$\pc(\cF_n)$ itself, since it is usually the other way round. 
An intuitive explanation is as follows: the down-set of $F$-free graphs~$\cF_n$ has the special property that its complement turns out to be an `easy' and well-understood up-set (i.e., the up-set of graphs containing~$F$), which is usually not the case. 
The main remaining open problem is to determine the order of magnitude\footnote{As pointed out by Jo{\~a}o Pedro de Abreu Marciano and Rajko Nenadov, using hypergraph container theorems one can easily obtain the upper bounds~$\qf(\cF) \le q(\cF) = O(n^{-1/m_2(F)}\log n)$, which demonstrates that our lower bounds~$q(\cF) \ge \qf(\cF) = \Omega(n^{-1/m_2(F)})$ are nearly best-possible.} of the (fractional) expectations thresholds~$q(\cF_n),\qf(\cF_n)$ for the down-set of~\mbox{$F$-free} graphs.

\bigskip{\noindent\bf Acknowledgements.} 
The results of this short note were proved after listening to Bhargav Narayanan's lecture on~\cite{GHN2021} 
during the Oberwolfach workshop “Combinatorics, Probability and Computing” in April~2022. 
We are grateful to the MFO institute for their hospitality and great working conditions.
We also thank Bhargav Narayanan, Jo{\~a}o Pedro de Abreu Marciano, Rajko Nenadov, Jinyoung Park and the referees for helpful correspondence and suggestions.

\small
\bibliographystyle{plain}


\normalsize

\end{document}